\documentclass{article}
\usepackage[latin1]{inputenc}
\usepackage[small]{titlesec}
\usepackage{amsmath}
\usepackage{amsthm}
\usepackage{amssymb}
\usepackage{graphicx}
\usepackage{standalone}
\usepackage{color}
\usepackage{array}
\theoremstyle{plain}
\newtheorem{thm}{Theorem}[section]

\newtheorem{lem}[thm]{Lemma}

\usepackage[small]{titlesec}
\usepackage{blindtext}
\usepackage{xcolor}

{\left\lbrace\begin{array}{@{}l@{}}}%
{\end{array}\right.}

\newcommand{\RNum}[1]{\uppercase\expandafter{\romannumeral #1\relax}}

\title{On a graph isomorphic to $NO^{+}(6,2)$}
\author{Federico Romaniello \footnote{Federico Romaniello:
federico.romaniello@unito.it
Dipartimento di Matematica "Giuseppe Peano"  -
Universit\`{a} di Torino - Via Carlo Alberto 10
- 10123 Torino (Italy),}\hfill\newline\hspace*{1.4em}
Valentino Smaldore \footnote{Valentino Smaldore:
valentino.smaldore@unipd.it 
Dipartimento di Tecnica e Gestione dei Sistemi Industriali - 
Universit\`{a} degli Studi di Padova - Stradella S. Nicola 3 - 36100 Vicenza (Italy).}}

\date{}

\begin{document}
\maketitle
\begin{abstract}
 Let $Q^{+}(2n-1,2)$ be a non-degenerate hyperbolic quadric of $PG(2n-1,2)$. Let $NO^{+}(2n,2)$ be the tangent graph, whose vertices are the points of $PG(2n-1,2) \setminus Q^{+}(2n-1,2)$ and two vertices $u,~v$ are adjacent if the line joining $u$ and $v$ is tangent to $Q^{+}(2n-1,2)$. Then $NO^{+}(2n-1,q)$ is a strongly regular graph. Let $\mathcal{V}^{4}_{2}$ be the \textit{Veronese surface} in $PG(5,q)$, and $\mathcal{M}^{3}_{4}$ its \textit{secant variety}. When $q=2$, $|Q^{+}(5,2)|=|\mathcal{M}^{3}_{4}|=35$. In this paper we define the graph $N\mathcal{M}^{3}_{4}$, with 28 vertices in $PG(5,2)\setminus\mathcal{M}^{3}_{4}$ and with the analogue incidence rule of the tangent graph. Such graph is isomorphic to $NO^{+}(6,2)$.

\end{abstract}
\section{Introduction}
A \textit{strongly regular graph} with parameters $(v,k,\lambda,\mu)$ is a graph with $v$ vertices where each vertex is incident with $k$ edges, any two adjacent vertices have $\lambda$ common neighbours, and any two non-adjacent vertices have $\mu$ common neighbours.
 Strongly regular graphs were introduced by R. C. Bose in \cite{Bose} in 1963, and ever since they have intensively been investigated. In particular, the eigenvalues of the adjacency matrix of a strongly regular are known; see \cite{brvan}:
     a strongly regular graph $G$ with parameters $(v,k,\lambda,\mu)$ has exactly three eigenvalues: $k$, $\theta_{1}$ and $\theta_{2}$ of multiplicity  $1$, $m_{1}$ and $m_{2}$, respectively, where:
    $$\theta_{1}=\frac{1}{2}\big[(\lambda-\mu)+\sqrt{(\lambda-\mu)^{2}+4(k-\mu)}\big],$$
    $$\theta_{2}=\frac{1}{2}\big[(\lambda-\mu)-\sqrt{(\lambda-\mu)^{2}+4(k-\mu)}\big],$$
    $$m_{1}=\frac{1}{2}\Big[(v-1)-\frac{2k-(v-1)(\lambda-\mu)}{\sqrt{(\lambda-\mu)^{2}+4(k-\mu)}}\Big],$$
    $$m_{2}=\frac{1}{2}\Big[(v-1)+\frac{2k-(v-1)(\lambda-\mu)}{\sqrt{(\lambda-\mu)^{2}+4(k-\mu)}}\Big].$$
   The \emph{spectrum} of a strongly regular graph is the triple $(k,\theta_{1}^{m_{1}},\theta_{2}^{m_{2}})$. Therefore, two strongly regular graphs with the same parameters are \textit{cospectral}, that is, they have the same spectrum. Two isomorphic graphs are always cospectral, but the converse is not always true. Indeed, even in the family of strongly regular graphs there are examples with same parameters, but not isomorphic.

   Several strongly regular graphs derive from finite polar spaces, in particular from collinearity graphs or incidence graphs. 
   The \textit{tangent graph} of a polar space $\mathcal{P}$ embedded in a projective space is defined to be the graph in which the vertices are the non-isotropic points with respect to the polarity defining $\mathcal{P}$, and two vertices are adjacent if and only if they lie on the same tangent line to $\mathcal{P}$. In \cite{brvan} it is proved that in some cases tangent graphs are strongly regular. In this paper we focus on the graph $NO^{+}(2n,2)$, belonging to a non-degenerate hyperbolic quadric $Q^{+}(2n-1,2)$.

  \section{The graph $NO^{+}(6,2)$}
   Consider a non-degenerate hyperbolic quadratic form in $PG(2n-1,2)$, the corresponding quadric $Q^{+}(2n-1,2)$ has $2^{2m-1}+2^{m-1}-1$ isotropic points. The graph $NO^{+}(2n,2)$ is the strongly regular graph with vertex set $PG(2n-1,2)\setminus Q^{+}(2n-1,2)$ where two vertices are adjacent if the points are orthogonal. $NO^{+}(2n,2)$ has the following parameters:\\
   $v=2^{2n-1}-2^{n-1},$\\
   $k=2^{2n-2}-1,$\\
   $\lambda=2^{2n-3}-2,$\\
   $\mu=2^{2n-3}+2^{n-2}$.\\
   The following lemma is a straightforward consequence of the construction of $NO^{+}(2n,2)$ since the group $P\Omega^{+}(2n,2)$ acts transitively on the external points of $Q^{+}(2n-1,2)$ and it preserves the adjacency properties of the graph.
  \begin{lem}
   \label{subgr}
   The simple group $P\Omega^{+}(2n,2)$ is an automorphism group of $NO^{+}(2n,2)$.
  \end{lem}

  In the case $n=3$ the parameters of $NO^{+}(6,2)$ are $(28,15,6,10)$ while the spectrum is $(15,-5^{7},1^{20})$. From Lemma \ref{subgr} we know that $P\Omega^{+}(6,2)\cong A_{8}$ acts on $NO^{+}(6,2)$ as an automorphism group. The full-automorphism group of the graph is $PGO^{+}(6,2)\cong S_{8}$ (see \cite{Svob}, Theorem 2).

  \section{The graph $N\mathcal{M}^{3}_{4}$}
  The \textit{Veronese surface} of all conics of $PG(2,q)$, is the variety $\mathcal{V}^{4}_{2}=\{(a^{2},b^{2},c^{2},ab,ac,bc)|(a,b,c)\in PG(2,q)\}\subseteq PG(5,q)$. The mapping
  $$\mu:\,
   \begin{cases}
    PG(2,q)\rightarrow PG(5,q)\\
   (x_{1},x_{2},x_{3})\mapsto(x_{1}^{2},x_{2}^{2},x_{3}^{2},x_{1}x_{2},x_{1}x_{3},x_{2}x_{3}).
   \end{cases}$$
   is called the Veronese embedding of $PG(2,q)$. The notation $\mathcal{V}^{2^n}_{n}=\mathcal{V}^{4}_{2}$ follows the one described in \cite[Notation 4.6]{2}. The variety $\mathcal{V}^{4}_{2}$ consists of $q^{2}+q+1$ points and its stabilising subgroup of $\mathcal{V}^{4}_{2}$ in $PGL(6,q)$ arises by \textit{lifting} from the group of collineations of $PG(2,q)$.
   \begin{thm}[{\cite[Proposition 4]{0}}]
    If $y\mapsto yA$ is a linear collineation of $PG(2,q)$ with $A=(a_{ij})$, $i,j=1,2,3$, the Veronese surface $\mathcal{V}^{4}_{2}$ is stabilised by the \emph{lifted} linear collineation $x\mapsto xB$ of $PG(5,q)$ given by the matrix
    $$B=\left(
      \begin{array}{cccccc}
        a_{11}^{2} & a_{12}^{2} & a_{13}^{2} & a_{11}a_{12} & a_{11}a_{13} & a_{12}a_{13} \\
        a_{21}^{2} & a_{22}^{2} & a_{23}^{2} & a_{21}a_{22} & a_{21}a_{23} & a_{22}a_{23} \\
        a_{31}^{2} & a_{32}^{2} & a_{33}^{2} & a_{31}a_{32} & a_{31}a_{33} & a_{32}a_{33} \\
        2a_{11}a_{21} & 2a_{12}a_{22} & 2a_{13}a_{23} & a_{11}a_{22}+a_{21}a_{12} & a_{11}a_{23}+a_{21}a_{13} & a_{12}a_{23}+a_{22}a_{13} \\
        2a_{11}a_{31} & 2a_{12}a_{32} & 2a_{13}a_{33} & a_{11}a_{32}+a_{31}a_{12} & a_{11}a_{33}+a_{31}a_{13} & a_{12}a_{33}+a_{32}a_{13} \\
        2a_{21}a_{31} & 2a_{22}a_{32} & 2a_{23}a_{33} & a_{21}a_{32}+a_{31}a_{22} & a_{21}a_{33}+a_{31}a_{23} & a_{22}a_{33}+a_{32}a_{23} \\
      \end{array}
    \right)$$
   \end{thm}
   The group of lifted collineations has the following orbits on the $q^{5}+q^{4}+q^{3}+q^{2}+q+1$ conics of $PG(2,q)$:
   \begin{itemize}
    \item $\mathcal{O}_{1}:=q^{2}+q+1$ double lines (points of $\mathcal{V}^{4}_{2}$);
    \item $\mathcal{O}_{2}:=\frac{1}{2}(q^{2}+q+1)(q^{2}-q)$ pairs of imaginary lines;
    \item $\mathcal{O}_{3}:=\frac{1}{2}(q^{2}+q+1)(q^{2}+q)$ pairs of intersecting lines;
    \item $\mathcal{O}_{4}:=q^{5}-q^{2}$ non-degenerate conics.
   \end{itemize}
   The set of all degenerate conics $\mathcal{O}_{1}\cup\mathcal{O}_{2}\cup\mathcal{O}_{3}=\mathcal{M}^{3}_{4}$ is called \textit{secant variety} and  $|\mathcal{M}^{3}_{4}|=|Q^{+}(5,q)|=(q^{2}+1)(q^{2}+q+1)$, see \cite[Theorem 4.18]{2}. The secant variety $\mathcal{M}^{3}_{4}$ is a hypersurface of degree 3 and dimension 4. We may identify points of $PG(5,q)$ with $3\times3$ symmetric matrices over $GF(q)$, by:
   $$(X_{1},X_{2},X_{3},X_{4},X_{5},X_{6})\longleftrightarrow\left(\begin{array}{ccc}
                                                                                        X_{1} & X_{4} & X_{5} \\
                                                                                        X_{4} & X_{2} & X_{6} \\
                                                                                        X_{5} & X_{6} & X_{3}
                                                                                      \end{array}\right).$$
   In this representation, the Veronese surface $\mathcal{V}^{4}_{2}$ correspond to the matrices $\left(\begin{array}{ccc}
                                                                                        x_{1}^{2} &x_{1}x_{2} & x_{1}x_{3} \\
                                                                                        x_{1}x_{2} & x_{2}^{2} & x_{2}x_{3} \\
                                                                                        x_{1}x_{3} & x_{2}x_{3} & x_{3}^{2}
                                                                                      \end{array}\right),$ while $\mathcal{M}^{3}_{4}$ is a cubic hypersurface with equation
   \begin{equation}
    \left|\begin{array}{ccc}
     X_{1} & X_{4} & X_{5} \\
     X_{4} & X_{2} & X_{6} \\
     X_{5} & X_{6} & X_{3}
    \end{array}\right|=0.
   \end{equation}
   With the above notation, the orbit $\mathcal{O}_{1}=\mathcal{V}^{4}_{2}$ coincides with the $3\times3$ symmetric matrices over $GF(q)$ of rank 1, while $\mathcal{O}_{2}$ and $\mathcal{O}_{3}$ with the $3\times3$ symmetric matrices over $GF(q)$ of rank 2, and $\mathcal{O}_{4}$ with the $3\times3$ symmetric matrices over $GF(q)$ of rank 3. For more details on the Veronese surface and the secant variety, see \cite{1} and \cite{2}. Moreover, the automorphism group and the orbits of its action on conics in $PG(2,q)$ may be further explored in \cite[Table 7.2]{Hirschfeld1}.

   \subsection{The case $q=2$}
    When $q=2$, the set $PG(5,2)\setminus\mathcal{M}^{3}_{4}$ has cardinality 28, as $|V(NO^{+}(6,2))|$, because $|Q^{+}(5,2)|=|\mathcal{M}^{3}_{4}|=35$. We now construct the graph $N\mathcal{M}^{3}_{4}$ with the same procedure of the tangent graph:
    \begin{itemize}
     \item $V(\mathcal{M}^{3}_{4})=PG(5,2)\setminus\mathcal{M}^{3}_{4}$;
     \item $E(\mathcal{M}^{3}_{4})=\{(x,y)|x,y\in V(\mathcal{M}^{3}_{4}),|\langle x,y\rangle\cap \mathcal{M}^{3}_{4}|=1\}.$
    \end{itemize}
    The lifted automorphism group is represented by the matrices
    $$B=\left(
      \begin{array}{cccccc}
        a_{11}^{2} & a_{12}^{2} & a_{13}^{2} & a_{11}a_{12} & a_{11}a_{13} & a_{12}a_{13} \\
        a_{21}^{2} & a_{22}^{2} & a_{23}^{2} & a_{21}a_{22} & a_{21}a_{23} & a_{22}a_{23} \\
        a_{31}^{2} & a_{32}^{2} & a_{33}^{2} & a_{31}a_{32} & a_{31}a_{33} & a_{32}a_{33} \\
        0 & 0 & 0 & a_{11}a_{22}+a_{21}a_{12} & a_{11}a_{23}+a_{21}a_{13} & a_{12}a_{23}+a_{22}a_{13} \\
        0 & 0 & 0 & a_{11}a_{32}+a_{31}a_{12} & a_{11}a_{33}+a_{31}a_{13} & a_{12}a_{33}+a_{32}a_{13} \\
        0 & 0 & 0 & a_{21}a_{32}+a_{31}a_{22} & a_{21}a_{33}+a_{31}a_{23} & a_{22}a_{33}+a_{32}a_{23} \\
      \end{array}
    \right),$$
    and the 4 orbits have the following properties:
    \begin{itemize}
     \item $|\mathcal{O}_{1}|=|\mathcal{O}_{2}|=7$, $|\mathcal{O}_{3}|=21$, $|\mathcal{O}_{4}|=28$;
     \item $\mathcal{O}_{2}$ is the nuclei plane $N:X_{1}=X_{2}=X_{3}=0$;
     \item $\mathcal{O}_{2}\cup\mathcal{O}_{4}$ is the Klein Quadric \\ $\mathcal{K}:X_{1}^{2}+X_{2}^{2}+X_{3}^{2}+X_{1}X_{2}+X_{2}X_{3}+X_{1}X_{3}+X_{1}X_{6}+X_{2}X_{5}+X_{3}X_{4}=0$.
    \end{itemize}
   
   It is possible to describe the graph $N\mathcal{M}^{3}_{4}$ in other ways and this is investigated deeply in the next subsection.  

  \subsection{Alternative descriptions of $N\mathcal{M}^{3}_{4}$}\label{diffrep}
    We give alternative descriptions of the graph $N\mathcal{M}^{3}_{4}$, each one highlighting a different property of it.\\
  
   Since $\mathcal{M}^{3}_{4}\cap\mathcal{K}=\mathcal{O}_{2}=N$, the secant variety always shares a plane with a Klein Quadric. Starting from the quadric $Q^{+}(5,2)$ in its canonical equation $X_{1}X_{6}+X_{2}X_{5}+X_{3}X_{4}=0$, we consider the Veronese hypersurface having exactly the plane $N:X_{1}=X_{2}=X_{3}=0$ in common with $Q^{+}(5,2)$. It is possible now to give another construction for the graph $N\mathcal{M}^{3}_{4}$, focusing on $Q^{+}(5,2):X_{1}X_{6}+X_{2}X_{5}+X_{3}X_{4}=0$ and $N:X_{1}=X_{2}=X_{3}=0$:
    \begin{itemize}
     \item $V(\mathcal{M}^{3}_{4})=Q^{+}(5,2)\setminus N$;
     \item $E(\mathcal{M}^{3}_{4})=\{(x,y)|x,y\in V(\mathcal{M}^{3}_{4}),|\langle x,y\rangle\cap N|=1\}\cup \{(x,y)|x,y\in V(\mathcal{M}^{3}_{4}),|\langle x,y\rangle\cap Q^{+}(5,2)|=2\}.$
    \end{itemize}

It is also possible to describe the graph in the representation as $3\times3$ matrices over $\mathbb{F}_{q}$:   
   \begin{itemize}
     \item $V(\mathcal{M}^{3}_{4})$ is the set of the non-singular symmetric matrices of order 3 over $\mathbb{F}_{q}$;
     \item $E(\mathcal{M}^{3}_{4})=\{(A,B)|A,B\in V(\mathcal{M}^{3}_{4}),A+B \mbox{ is singular }\}.$
    \end{itemize}
  
Moreover, observing that $(\mathcal{M}^3_4)^C \cup N = Q^+(5,2)$, where  $(\mathcal{M}^3_4)^C$ is the complement of $\mathcal{M}^3_4$, the vertices of the graph can also be described as the points of $Q^{+}(5,2)\setminus N$, and two vertices $u$ and $v$ are adjacent if $u\not\in v^{\perp}$ or $u\in\langle v,v^{\perp}\cap N\rangle$, where $\perp$ is the polarity induced by $Q^{+}(5,2)$.\\
  
In what follows we will use the latter description of $N\mathcal{M}^{3}_{4}$, as it is the most convenient one for the analysis of the strong regularity of the graph (cf. Section \ref{srgm}).

\section{Strong regularity of $N\mathcal{M}^{3}_{4}$}\label{srgm}

In this section we will focus on the strong regularity of $N\mathcal{M}^{3}_{4}$, using the latter representation described in Subsection \ref{diffrep}.
\begin{thm}
 The graph $N\mathcal{M}^{3}_{4}$ is strongly regular with parameters $(v,k,\lambda,\mu)=(28,15,6,10)$.	
\end{thm}

\begin{proof}
Each of the 28 vertices lies on 9 isotropic lines. Fix $P\in Q^{+}(5,2)\setminus N$ The polar hyperplane $P^{\perp}$ meets the quadric in a cone having vertex in $P$ and a hyperbolic quadric $Q^{+}(3,2)$ as base. Among the 9 isotropic lines through $P$, 3 meet the nuclei plane in an isotropic line of $Q^{+}(3,2)$. Hence, we have 3 neighbours so represented. The other 6 lines define 12 non-neighbours of $P$, and all the other 12 points generate a secant line with $P$, so that the graph is $15$-regular. Given two vertices $u$ and $v$, to get the parameters $\lambda$ and $\mu$, we may consider separately the two type of adjacencies, since $uv$ is an edge if $u\not\in v^{\perp}$ or $u\in\langle v,v^{\perp}\cup N\rangle$. As $u^{\perp}$ meets $Q^{+}(5,2)$ in a cone with vertex $P$ and base $Q^{+}(3,2)$, we have that $u^{\perp} \cap N$ is a line, and $\langle u,u^{\perp}\cup N\rangle$ is a plane. In what follows, $\pi_{u}$ stands for the projective plane $\pi_{u} =\langle u,u^{\perp}\cup N\rangle$.\\

We start by calculating $\mu$, the number of common neighbours of non-adjacent vertices. Let $w \in Q^{+}(5,2)\setminus N$ be a common neighbour of the non-adjacent vertices $u$ and $v$, with $u\in v^{\perp}\setminus\pi_{v}$ and $v\in u^{\perp}\setminus\pi_{u}$.
Hence, we must distinguish different cases:\\
\fbox{Case $1$} \textbf{$w\in\pi_{u}$ and $w\in\pi_{v}$}\\
 Since $\langle u,N\rangle$ and $\langle v,N\rangle$ are two solids contained in the 4-space $\langle u,v,N\rangle$, their intersection is the plane $N$. Therefore, $\pi_{u}\cap\pi_{v}\subseteq N$, and we have no such common neighbours.\\
\fbox{Case $2$} \textbf{$w\in\pi_{u}$ and $w\notin v^{\perp}$} or \textbf{$w\not\in u^{\perp}$ and $w\in\pi_{v}$}\\
We will investigate only the first subcase, the second one is analogous by switching the role of $u$ and $v$. Suppose that $w\in\pi_{u}$ and $w\notin v^{\perp}$, then we have exactly two such common neighbours. Hence $|\pi_{u}|=7$ and $|\pi_{u}\setminus N|=4$. Moreover the projective plane $\pi_{v}$ meets $\pi_{u}$ in a line, passing through $u$ and meeting $N$ in a point, so $|(\pi_{u}\setminus (N\cup v^{\perp}))|=2$.\\
\fbox{Case $3$} \textbf{$w\not\in u^{\perp}$ and $w\notin v^{\perp}$}\\
 Since $u^{\perp}\cap v^{\perp}=\langle u,v\rangle^{\perp}$, it is a solid that meets the quadric in a cone having as vertex the line $\langle u,v\rangle$ and base a $Q^{+}(1,2)$, and so $|u^{\perp}\cap v^{\perp}|$ has 11 points on the quadric $Q^{+}(5,2)$. Moreover, $u^{\perp}$ cuts a cone with vertex $u$ and base a $Q^{+}(3,2)$, and $u^{\perp}$ has 19 points on $Q^{+}(5,2)$. Hence we get 27 isotropic points on $u^{\perp}\cup v^{\perp}$, while the space $(u^{\perp}\cup v^{\perp})\cap N$ consists on two concurrent lines. Then, the number of common neighbours in Case 3 is equal to $28-|(u^{\perp}\cup v^{\perp})\setminus N|=28-(27-5)=6$.
\\
In conclusion $\mu=0+2+2+6=10$.\\

To calculate the value of $\lambda$ we remark that two adjacent vertices may lie in both a secant line, or in an isotropic line having the third point in $N$.
Let $w \in Q^{+}(5,2)\setminus N$ be a common neighbour of the adjacent vertices $u$ and $v$, and let $u\in\pi_{v}$ and $v\in\pi_{u}$.

Different cases arises:\\
\fbox{Case $1$} \textbf{$w\in\pi_{u}$ and $w\in\pi_{v}$}\\
 In this case $\pi_{u}=\pi_{v}$ and we have exactly two more common neighbours in $\pi_{u}\setminus N$\\
\fbox{Case $2$} \textbf{$w\in\pi_{u}$ and $w\notin v^{\perp}$} or \textbf{$w\not\in u^{\perp}$ and $w\in\pi_{v}$}\\
 We will investigate only the first subcase, the second one is analogous by switching the role of $u$ and $v$. Suppose that $w\in\pi_{u}$ and $w\notin v^{\perp}$. In this case $\pi_{u}\subseteq v^{\perp}$, and there are no common neighbours.\\
\fbox{Case $3$} \textbf{$w\not\in u^{\perp}$ and $w\notin v^{\perp}$}\\
 As before we have 27 isotropic points on $u^{\perp}\cup v^{\perp}$, but the space $(u^{\perp}\cup v^{\perp})\cap N$ consists on the line $\pi_{u}\cap N$. Then, the number of common neighbours in Case 3 is equal to $28-|(u^{\perp}\cup v^{\perp})\setminus N|=28-(27-3)=4$.
\\
The number of common neighbours is $2+0+0+4=6$.\\

Let now $w \in Q^{+}(5,2)\setminus N$ be a common neighbour of the adjacent vertices $u$ and $v$, and let $u\not\in v^{\perp}$ and $v\not\in u^{\perp}$.

Again, we must distinguish different cases:\\
\fbox{Case $1$} \textbf{$w\in\pi_{u}$ and $w\in\pi_{v}$}\\
 Since $\langle u,N\rangle$ and $\langle v,N\rangle$ are two solids contained in the 4-space $\langle u,v,N\rangle$, their intersection is the plane $N$. Therefore, $\pi_{u}\cap\pi_{v}\subseteq N$, and we have no such common neighbours.\\
\fbox{Case $2$} \textbf{$w\in\pi_{u}$ and $w\notin v^{\perp}$} or \textbf{$w\not\in u^{\perp}$ and $w\in\pi_{v}$}\\
 We will investigate only the first subcase, the second one is analogous by switching the role of $u$ and $v$. Suppose that $w\in\pi_{u}$ and $w\notin v^{\perp}$, then we have exactly one such common neighbour. Firstly $|\pi_{u}|=7$, and $v^{\perp}$ meets the plane $\pi_{u}$ in a line $\ell$. Then the common neighbours are exactly $|\pi_{u}\setminus(N\cup\ell\cup\{u\})|=1$.\\
\fbox{Case $3$} \textbf{$w\not\in u^{\perp}$ and $w\notin v^{\perp}$}\\
 $u^{\perp}$ and $v^{\perp}$ are two 4-spaces meeting the quadric in two cones with the point as vertex and a $Q^{+}(3,2) $ as base. Since $u^{\perp}\cap v^{\perp}$ is a solid, $u^{\perp}\cap v^{\perp}$ meet the quadric in 9 points. Hence we get 29 isotropic points on $u^{\perp}\cup v^{\perp}$, while the space $(u^{\perp}\cup v^{\perp})\cap N$ consists on two concurrent lines. Then,  the number of common neighbours in Case 3 is equal to $28-|(u^{\perp}\cup v^{\perp})\setminus N|=28-(29-5)=4$.
\\
Even in this case the number of common neighbours is $0+1+1+4=6$, thus we can conclude that $\lambda=6$.
\end{proof}

\section{The isomorphism issue}
By a classical result in graph theory, two graphs are isomorphic if and only if their adjacency matrices are similar, see \cite{Biggs}. There are 4 known non-isomorphic strongly regular graphs with parameters $(28,15,6,10)$ in the Spence's database (see \cite{Spence}). Looking at the adjacency matrices of $NO^{+}(6,2)$ and $N\mathcal{M}^{3}_{4}$ it is straightforward to check the similarity between the two matrices.

One isomorphism between these two graphs is given by the following:
$$(1,0,1,0,1,1)\longmapsto(1,1,0,0,1,1)$$
$$(1,1,1,0,0,1)\longmapsto(0,0,1,0,0,0)$$
$$(0,1,1,0,1,1)\longmapsto(1,1,1,1,0,1)$$
$$(0,1,0,1,1,0)\longmapsto(1,1,0,1,0,0)$$
$$(1,0,0,1,0,1)\longmapsto(0,1,0,0,0,1)$$
$$(0,0,1,1,0,0)\longmapsto(0,1,1,0,0,1)$$
$$(1,1,0,1,1,0)\longmapsto(1,1,1,0,1,1)$$
$$(1,0,1,1,0,0)\longmapsto(0,1,0,1,0,0)$$
$$(1,0,0,1,1,1)\longmapsto(1,0,1,1,0,1)$$
$$(0,0,1,1,1,0)\longmapsto(1,0,0,1,1,0)$$
$$(1,1,0,1,0,1)\longmapsto(0,1,0,0,0,0)$$
$$(0,1,1,1,0,0)\longmapsto(0,1,1,0,0,0)$$
$$(0,1,0,1,1,1)\longmapsto(1,0,0,0,1,0)$$
$$(1,0,1,1,1,0)\longmapsto(1,0,1,0,0,0)$$
$$(1,1,1,1,0,0)\longmapsto(0,1,0,1,0,1)$$
$$(0,0,1,1,0,1)\longmapsto(0,0,1,0,1,0)$$
$$(0,0,1,1,1,1)\longmapsto(1,1,0,0,0,0)$$
$$(0,1,1,1,0,1)\longmapsto(0,0,1,0,1,1)$$
$$(0,1,0,0,1,0)\longmapsto(1,1,1,0,0,0)$$
$$(1,0,0,0,0,1)\longmapsto(0,1,1,1,1,1)$$
$$(1,1,0,0,1,0)\longmapsto(1,1,0,1,1,1)$$
$$(1,1,1,1,1,1)\longmapsto(1,1,1,1,1,0)$$
$$(1,0,0,0,1,1)\longmapsto(1,0,0,0,0,0)$$
$$(1,1,0,0,0,1)\longmapsto(0,1,1,1,1,0)$$
$$(0,1,0,0,1,1)\longmapsto(1,0,1,1,1,1)$$
$$(0,1,1,0,1,0)\longmapsto(1,0,1,0,1,0)$$
$$(1,0,1,0,0,1)\longmapsto(0,0,1,0,0,1)$$
$$(1,1,1,0,1,0)\longmapsto(1,0,0,1,0,0)$$

Since adjacencies in both $NO^{+}(6,2)$ and $N\mathcal{M}^{3}_{4}$ are defined also in the nuclei plane $N$, we now define the graphs $\widehat{NO^{+}(6,2)}$ and $\widehat{N\mathcal{M}^{3}_{4}}$, with the same vertex sets of $NO^{+}(6,2)$ and $N\mathcal{M}^{3}_{4}$ and such that in both cases two vertices $u$ and $v$ are adjacent if and only if the third point on the line $uv$ is in $N$.
\begin{lem}
 $\widehat{N\mathcal{M}^{3}_{4}}\cong\widehat{NO^{+}(6,2)}$ and they both consists of 7 copies of $K_{4}$.
\end{lem}
\begin{proof}
 Firstly, we prove that $\widehat{N\mathcal{M}^{3}_{4}}$ and $\widehat{NO^{+}(6,2)}$ are $3$-regular.

  Let $P \in V(\widehat{N\mathcal{M}^{3}_{4}})$, then $P^{\perp}$ meets $Q^{+}(5,2)$ in a cone with vertex $P$ and base $Q^{+}(3,2)$, hence $P^{\perp} \cap N$ is a line, say $\ell$. Now suppose that $P \in V(\widehat{NO^{+}(6,2)})$, we have that $P^{\perp}$ meets $Q^{+}(5,2)$ in a parabolic quadric $Q(4,2)$, and also in this case $P^{\perp} \cap N$ is a line, say $\ell^{\prime}$. Thus the neighbours of $P$ are the three point on the line from $P$ to $\ell$ in the former case, and to $\ell^{\prime}$ in the latter.

Now, the space joining $P$ and $\ell$ is a projective plane $ \pi= \langle P,\ell \rangle$, made of 7 points: $P$, the three points of $\ell$, and the three neighbours of $P$. Let $R$ be one of the neighbours of $P$. The plane $\langle R,\ell \rangle$ meets $\pi$ in $5$ points: $P$, $R$ and the three points on $\ell$. Thus $\langle R,\ell \rangle$ must be $\pi$, and $P$ and its neighbours are a connected component isomorphic to $K_4$ in $\widehat{N\mathcal{M}^{3}_{4}}$. The same argument holds also for $\langle P, \ell^{\prime} \rangle$ and $\widehat{NO^{+}(6,2)}$ is made of connected components isomorphic to $K_4$.
\end{proof}
We can see that each connected component $K_{4}$ correspond uniquely to a line of $N$, and an isomorphism between the graphs $\widehat{NO^{+}(6,2)}$ and $\widehat{N\mathcal{M}^{3}_{4}}$ sends one connected $K_{4}$ of $\widehat{NO^{+}(6,2)}$ in one connected component of $\widehat{N\mathcal{M}^{3}_{4}}$.

\section{Conclusion}
 It is shown that the two strongly regular graphs $NO^{+}(6,2)$ and $N\mathcal{M}^{3}_{4}$, arising from the Klein Quadric and the Veronese hypersurface, are isomorphic. By the way, even for $q>2$, $|\mathcal{M}^{3}_{4}|=|Q^{+}(5,q)|=(q^{2}+1)(q^{2}+q+1)$. It is natural to ask what is the connection between this two objects, in a pure geometrical point of view. Moreover, further researches should show whether there is a connection between the \textit{Klein representation} of the $(q^{2}+1)(q^{2}+q+1)$ lines of $PG(3,q)$ and the \textit{Veronese embedding} of the $(q^{2}+1)(q^{2}+q+1)$ (degenerate) conics of $PG(2,q)$.

\section*{Acknowledgements}
 The authors are grateful to the anonymous reviewers for their hints to improve the quality and the readability of the paper. The authors also wish to thank A. Cossidente for fruitful discussion on the Veronese surface and the related combinatorial structures.

\end{document}